\documentclass{amsart}
\usepackage{amsfonts}
\usepackage{latexsym}
\usepackage{amssymb}
\usepackage{amsmath}
\usepackage{color}

\newcommand{\R}{\mathbb R}

\newtheorem{thm}{Theorem}[section]
\newtheorem{lemma}[thm]{Lemma}
\newtheorem{proposition}[thm]{Proposition}

\theoremstyle{remark}
\newtheorem{rmk}{Remark}

\begin{document}


\title[Berwald's inequality and its applications]{An extension of Berwald's inequality and its relation to Zhang's inequality}

\author[D. Alonso]{David Alonso-Guti\'{e}rrez}
\address{\'Area de An\'alisis Matem\'atico, Departamento de Matem\'aticas, Facultad de Ciencias, Universidad de Zaragoza, Pedro Cerbuna 12, 50009 Zaragoza (Spain), IUMA}
\email[(David Alonso-Guti\'errez)]{alonsod@unizar.es}
\thanks{The first and second authors are partially supported by MINECO Project MTM2016-77710-P, DGA E26\_17R and IUMA.
This research is a result of the activity developed within the framework of the Programme in Support of Excellence Groups of the Regi\'on de Murcia, Spain, by Fundaci\'on S\'eneca, Science and Technology Agency of the Regi\'on de Murcia. The third author is partially supported by Fundaci\'on S\'eneca project 19901/GERM/15, Spain, and by MINECO Project MTM2015-63699-P Spain.}
\date{\today}

\author[J.\,Bernu\'es]{Julio Bernu\'es}
\address{\'Area de an\'alisis matem\'atico, Departamento de matem\'aticas, Facultad de Ciencias, Universidad de Zaragoza, Pedro cerbuna 12, 50009 Zaragoza (Spain), IUMA}
\email[(Julio Bernu\'es)]{bernues@unizar.es}

\author[B. Gonz\'alez]{Bernardo Gonz\'alez Merino}
\address{Departamento de An\'alisis Matem\'atico, Facultad de Matem\'aticas, Universidad de Sevilla, Apdo. 1160, 41080-Sevilla, Spain}
\email[Bernardo Gonz\'alez Merino]{bgonzalez4@us.es}
\begin{abstract}
In this note prove the following Berwald-type inequality, showing
that for any integrable log-concave function $f:\mathbb R^n\rightarrow[0,\infty)$ and any concave function $h:L\rightarrow\mathbb [0,\infty)$, where $L$  is the epigraph of $-\log \frac{f}{\Vert f\Vert_\infty}$, then
	$$p\to
	\left(\frac{1}{\Gamma(1+p)\int_L e^{-t}dtdx}\int_L h^p(x,t)e^{-t}dtdx\right)^\frac{1}{p}
	$$
is decreasing in $p\in(-1,\infty)$, extending the range of $p$ where the monotonicity is known to hold true.

As an application of this extension, we will provide a new proof of a functional form of Zhang's reverse Petty projection inequality, recently obtained in \cite{ABG}.
\end{abstract}
\maketitle
\section{Introduction and notation}

Let $K\subseteq\R^n$ be a \textit{convex body}, i.e., a compact, convex set with non-empty interior, and let us denote by $\mathcal K^n$ the set of all convex bodies in $\R^n$ and by $|K|$ the \textit{Lebesgue measure} of $K$. We will also denote by $\mathcal{K}_0^n$ the set of convex bodies containing the origin. It is well known that, as a consequence of H\"older's inequality, for any integrable function $f:K\rightarrow[0,\infty)$ the function
\[
p\to\left(\frac{1}{|K|}\int_Kf(x)^pdx\right)^\frac1p
\]
is increasing in $p\in(0,\infty)$.

A famous inequality proved by Berwald
\cite[Satz 7]{Ber} (see also \cite[Theorem 7.2]{AAGJV} for a translation into English) provides a reverse H\"older's inequality for $L_p$-norms ($p>0$) of concave functions defined on convex bodies. It states that for any $K\in\mathcal K^n$  and any concave function $f:K\to[0,\infty)$ ,
then
\begin{equation}\label{thm:GeometricBerwald}
p\to\left(\frac{{p+n\choose n}}{|K|}\int_K f(x)^{p}\,dx\right)^\frac{1}{p}
\end{equation}
is decreasing in $p\in(0,\infty)$.

A function $f:\mathbb R^n\rightarrow[0,\infty)$ is called log-concave if for every $x,y\in\R^n, 0<\lambda<1$, $f(\lambda x+(1-\lambda)y)\ge (f(x))^{\lambda}(f(y))^{1-\lambda}$. Throughout the paper, we will denote by $\mathcal F(\mathbb R^n)$ be the set of all integrable log-concave functions in $\mathbb R^n$.


In the context of log-concave functions, the following version of Berwald's inequality \eqref{thm:GeometricBerwald} on epigraphs of convex functions was proved in \cite[Lemma 3.3]{AAGJV}:

\textit{``Let $f\in\mathcal{F}(\R^n)$ and let $h:L\to[0,\infty)$ be a continuous concave non-identically null function, where $L=\{(x,t)\in\R^{n+1}\,:\,f(x)\geq e^{-t}\Vert f\Vert_\infty\}$ is the the epigraph of $-\log \frac{f}{\Vert f\Vert_\infty}$. Then, the function}
\begin{equation}\label{thm:BerwadFunctionalSmallRange}
p\to\left(\frac{1}{\Gamma(1+p)\int_L e^{-t}dtdx}\int_L h^p(x,t)e^{-t}dtdx\right)^\frac{1}{p}
\end{equation}
\textit{is decreasing in $p\in(0,\infty)$."}

When providing a new proof of Zhang's reverse Petty projection inequality, Gardner and Zhang \cite{GZ} extended \eqref{thm:GeometricBerwald} to the larger range of values $p>-1$ (see \cite[Theorem 5.1]{GZ}). The first goal in this paper is to also extend \eqref{thm:BerwadFunctionalSmallRange} to the larger range of values $p>-1$. 

\begin{thm}\label{thm:FuncBerwald>-1}
Let $f\in\mathcal F(\mathbb R^n)$ and let $h:L\to[0,\infty)$ be a concave function,
where $L=\{(x,t)\in\R^{n+1}\,:\,f(x)\geq e^{-t}\Vert f\Vert_\infty\}$. Then, the function
	$$p\to
	\left(\frac{1}{\Gamma(1+p)\int_L e^{-t}dtdx}\int_L h^p(x,t)e^{-t}dtdx\right)^\frac{1}{p}
	$$
	is decreasing in $p\in(-1,\infty)$.
\end{thm}

For any $K\in\mathcal K^n$, its \textit{polar projection body} $\Pi^*(K)$ is the unit ball of the norm given by
$$
\Vert x\Vert_{\Pi^*(K)}:=|x||P_{x^\perp}K|,\qquad x\in\R^n
$$
where $P_{x^\perp}K$ is the \textit{orthogonal projection} of $K$  onto the hyperplane orthogonal to $x$, $|\cdot|$ denotes (besides the Lebesgue measure in the suitable space) the Euclidean norm and $\Vert \cdot\Vert_K$ denotes the \textit{Minkowski functional} of $K$, defined for every $x\in\R^n$, as $\Vert x\Vert_K:=\inf\{\lambda>0\mid x\in\lambda K\}\in[0,\infty]$. It is a norm if and only if $K$ is centrally symmetric.

The expression $|K|^{n-1}|\Pi^*(K)|$ is affine invariant and its extremal convex bodies are well known: \textit{Petty's projection inequality} \cite{P} states that the (affine class of the) $n$-dimensional Euclidean ball, $B_2^n$, is the only maximizer and \textit{Zhang's inequality} \cite{Z1} proves that the (affine class of the)  $n$-dimensional simplex $\Delta_n$, is the only minimizer. That is, for any convex body $K\subseteq\R^n$,
\begin{equation}\label{eq:PettyAndZhang}
\frac{\binom{2n}{n}}{n^n}=|\Delta_n|^{n-1}|\Pi^*(\Delta_n)|\leq|K|^{n-1}|\Pi^*(K)|\leq |B_2^n|^{n-1}|\Pi^*(B_2^n)|=\frac{|B_2^n|^n}{|B_2^{n-1}|^n}.
\end{equation}

In recent years, many relevant geometric inequalities have been extended to the general context of log-concave functions (see for instance \cite{AKM}, \cite{KM}, \cite{C}, or \cite{HJM} and the references therein).
Let us recall that $\mathcal K^n$ and $\mathcal{K}_0^n$ naturally embed into $\mathcal F(\mathbb R^n)$, via the natural injections
\[
K\to\chi_K\quad\text{ and }\quad K\to e^{-\Vert\cdot\Vert_K},
\]
where $\chi_K$ is the characteristic function of $K$. These and other basic facts on convex bodies and log-concave functions used in the paper can be found in \cite{BGVV} and \cite{AGM}.

For any $f\in\mathcal F(\mathbb R^n)$, the \textit{polar projection body of $f$}, denoted as $\Pi^*(f)$,
is the unit ball of the norm given by
$$\Vert x\Vert_{\Pi^*(f)}:=2|x|\Vert f\Vert_{\infty}\!\int_{0}^{\infty}\!|P_{x^\perp}K_t(f)|\ e^{-t}dt=
2\Vert f\Vert_{\infty}\int_{0}^{\infty}\!\Vert x\Vert_{\Pi^*(K_t(f))}e^{-t}dt,$$
where $K_t(f):=\{x\in\R^n\,:\, f(x)\geq e^{-t}\Vert f\Vert_\infty\}$,  $t>0$ (see \cite{AGJV}).

In \cite{ABG}, an extension of Zhang's inequality (i.e.,~the left hand side inequality in \eqref{eq:PettyAndZhang})
was proved in the settings of log-concave functions.
\begin{thm}\label{thm:FunctionalZhang} Let $f\in\mathcal F(\mathbb R^n)$. Then,
\begin{equation}\label{eq:FunctionalZhang}
	\int_{\R^n}\int_{\R^n}\min\left\{f(y),f(x)\right\}dydx\leq 2^nn!\Vert f\Vert_1^{n+1}\left|\Pi^*\left(f\right)\right|.
\end{equation}
 Moreover, if $\Vert f\Vert_\infty=f(0)$ then equality holds if and only if $\frac{f(x)}{\Vert f\Vert_\infty}=e^{-\Vert x\Vert_{\Delta_n}}$ for some n-dimensional simplex $\Delta_n$ containing the origin.
\end{thm}
Observe that when $f=e^{-\Vert\cdot\Vert_K}$ for some convex body $K\in\mathcal{K}_0^n$, then \eqref{eq:FunctionalZhang} recovers Zhang's inequality.

Our second goal here is to provide a new proof of the functional version of Zhang's inequality \eqref{eq:FunctionalZhang} by using the extension of Berwald's inequality given by Theorem \ref{thm:FuncBerwald>-1}, in a similar way as Gardner and Zhang  \cite{GZ} proved the geometrical version of Zhang's inequality via their extension of Berwald's inequality \eqref{thm:GeometricBerwald} to $p>-1$.

A common feature in both proofs, the one given in \cite{ABG} and the one in this paper, is the crucial role played by the functional form of the covariogram function $g_f$ associated to the function  $f\in\mathcal{F}(\R^n)$.  See \cite{ABG} and its definition below. Recall that in the geometric setting the covariogram function of a convex body $K$ is given by $g_K(x)=|K\cap (x+K)|$. Apart from this fact, the two proofs completely differ.

We introduce further notation: $S^{n-1}$ denotes the Euclidean unit sphere in $\R^n$.
If the origin is in the interior of a convex body $K$, the function $\rho_K\colon S^{n-1}\to [0,+\infty)$ given by $\rho_K(u)=\sup\{\lambda\ge 0\mid \lambda u\in K\}$ is the \textit{radial function} of $K$. It extends to $\R^{n}\setminus\{0\}$ via $t\rho_K(tu)=\rho_K(u)$, for any $t>0,u\in S^{n-1}$.

Finally, for any function $f\in\mathcal{F}(\R^n)$ let $g_f$ be the covariogram functional of $f$, is defined by
$$
g_f(x):=\int_0^\infty e^{-t}|K_t(f)\cap(x+ K_t(f))|dt
$$
(cf.~\cite{ABG}).

The paper is organized as follows: Section \ref{sec:Berwald} contains the aforementioned extension, Theorem \ref{thm:FuncBerwald>-1},
of the functional Berwald inequality to the larger range of values of $p>-1$.
In Section \ref{sec:Proof} we recall the celebrated family (with parameter $p>0$) of convex bodies associated to any log-concave function introduced by Ball in \cite[pg.~74]{B}. We also recall the properties of the covariogram functional of a log-concave function, proven in \cite{ABG}.
Another main ingredient in the proof in \cite{GZ} is an expression that connects the covariogram function of a convex body $K$ and Ball's convex bodies. Such a connection can be extended to the functional form of the covariogram $g_f$ of a log-concave function and moreover, the polar projection body of $f$ will appear as a limiting case of this new expression when the value of the parameter $p$ tends to $-1$.

\section{An extension of Berwald's inequality}\label{sec:Berwald}

In this section we will prove the aforementioned extension of Berwald's inequality, see Theorem \ref{thm:FuncBerwald>-1} above. We first state a 1-dimensional lemma that can be seen as a degenerate version of Theorem \ref{thm:FuncBerwald>-1}.

\begin{lemma}\label{1dim}
	Let $\gamma:[0,\infty)\to[0,\infty)$ be a non-decreasing concave function and define
	$$
	\Phi_\gamma(p)=\left(\frac{1}{\Gamma(1+p)}\int_0^\infty \gamma(r)^p e^{-r}dr\right)^\frac{1}{p}, \, p>-1.
	$$
Then $\Phi_\gamma(p)$ is decreasing in $p$ in $(-1,\infty)$. Furthermore, if there exist $-1<p_1<p_2$ such that $\Phi_\gamma(p_1)=\Phi_\gamma(p_2)$, then $\gamma$ is a linear function and $\Phi_\gamma$ is constant on $(-1,\infty)$.
	\end{lemma}
	
\begin{rmk}
As usual, we define $\Phi_\gamma(0)=\lim_{p\to 0}\Phi_\gamma(p)$ which by straightforward computations
(using L'H$\hat{\text{o}}$pital's rule, interchanging the integral and the derivative operations,
and taking into account that $\frac{\partial\Gamma(1+x)}{\partial x}|_{x=0}=-A$, where $A\approx 0.577$ is the Euler-Mascheroni constant) yields
$\Phi_\gamma(0)=e^{A} \exp\Big(\int_{0}^{\infty}\log \gamma(r) e^{-r}dr\Big)$. 
\end{rmk}

\begin{proof}[Proof of Lemma \ref{1dim}]
Fix $0\ne p_1>-1$ and
write  $\overline{\gamma}(r)=\Phi_{\gamma}(p_1)\cdot r,\ r\ge 0$.
For any $p>-1,$
	$$
	\Phi_{\overline{\gamma}}(p)=\left(\frac{1}{\Gamma(1+p)}\int_0^\infty \Phi_{\gamma}(p_1)^p r^p e^{-r}dr\right)^\frac{1}{p}=\Phi_{\gamma}(p_1).
	$$
	Therefore
	\begin{equation}\label{eq:IntegralEquals0}
	0=\Phi_{\gamma}^{p_1}(p_1)-\Phi_{\overline{\gamma}}^{p_1}(p_1)=\frac{1}{\Gamma(1+p_1)}\int_0^\infty(\gamma(r)^{p_1}-\overline{\gamma}(r)^{p_1})e^{-r}dr,
	\end{equation}
	or equivalently,
	$$
	\int_0^1(\gamma(-\log t)^{p_1}-\overline{\gamma}(-\log t)^{p_1})dt=0.$$
	
	We first consider the case $-1<p_1<p_2<0$.

	Since the function $\gamma$ is non-negative and concave and \eqref{eq:IntegralEquals0} holds, if $\gamma$ is not identically equal to $\overline{\gamma}$, i.e., $\gamma$ is not linear, there exists a unique $r_0\in(0,\infty)$ such that $\gamma(r)>\overline{\gamma}(r)$ if $r\in(0,r_0)$ and $\gamma(r)<\overline{\gamma}(r)$ if $r\in(r_0,\infty)$. Denoting $t_0=e^{-r_0}$, we have that $\gamma(-\log t)<\overline{\gamma}(-\log t)$ if $t\in(0,t_0)$ and $\gamma(-\log t)>\overline{\gamma}(-\log t)$ if $t\in(t_0,1)$.	Now,
	\begin{eqnarray*}
		\Gamma(1+p_2)(\Phi_{\gamma}^{p_2}(p_2)-\Phi_{\overline{\gamma}}^{p_2}(p_2))&=&\int_0^\infty(\gamma(r)^{p_2}-\overline{\gamma}(r)^{p_2})e^{-r}dr\cr
		&=&\int_0^1(\gamma(-\log t)^{p_2}-\overline{\gamma}(-\log t)^{p_2})dt\cr
		&=&\int_0^1(\gamma(-\log t)^{p_1}-\overline{\gamma}(-\log t)^{p_1})\psi(t)dt,\cr
	\end{eqnarray*}
	where
	$$
	\psi(t)=\frac{\gamma(-\log t)^{p_2}-\overline{\gamma}(-\log t)^{p_2}}{\gamma(-\log t)^{p_1}-\overline{\gamma}(-\log t)^{p_1}}.
	$$
Since $w(x)=x^\frac{p_2}{p_1}$ is strictly concave in $(0,\infty)$,
	$\displaystyle
	\frac{w(x)-w(y)}{x-y}
	$
	is strictly decreasing in $x$ and $y$ and, since  $\gamma(-\log t)^{p_1}$ is non-decreasing and $\overline{\gamma}(-\log t)^{p_1}$ is strictly increasing in $t$, $\psi(t)$ is strictly decreasing. Now, by the mean value theorem, there exist $c_1\in(0, t_0)$ and $c_2\in(t_0,1)$ such that
	\begin{eqnarray*}
		&&\!\!\!\!\int_0^1(\gamma(-\log t)^{p_1}-\overline{\gamma}(-\log t)^{p_1})\psi(t)dt\cr
		&=&\!\!\!\!\int_0^{t_0}(\gamma(-\log t)^{p_1}-\overline{\gamma}(-\log t)^{p_1})\psi(t)dt+\int_{t_0}^1(\gamma(-\log t)^{p_1}-\overline{\gamma}(-\log t)^{p_1})\psi(t)dt\cr
		&=&\!\!\!\!\psi(c_1)\!\int_0^{t_0}\!\!(\gamma(-\log t)^{p_1}\!-\overline{\gamma}(-\log t)^{p_1})dt
		+\psi(c_2)\!\int_{t_0}^1\!(\gamma(-\log t)^{p_1}\!-\overline{\gamma}(-\log t)^{p_1})dt\cr
		&=&\!\!\!\!(\psi(c_1)-\psi(c_2))\int_0^{t_0}(\gamma(-\log t)^{p_1}-\overline{\gamma}(-\log t)^{p_1})dt>0 ,\cr
	\end{eqnarray*}
	since $\psi$ is strictly decreasing, $\gamma(-\log t)<\overline{\gamma}(-\log t)$ for $t\in(0,t_0)$ and $p_1<0$. Therefore, if $\gamma$ is not linear, $\Phi_{\gamma}(p_2)<\Phi_{\overline{\gamma}}(p_2)=\Phi_{\overline{\gamma}}(p_1)=\Phi_{\gamma}(p_1)$.
	
	The case $0<p_1<p_2$ follows analogously with straightforward changes (in this case, if $\gamma$ is not linear $w$ is strictly convex and $\psi$ is strictly decreasing). The continuity of $\Phi_{\gamma}$ in $0$ then implies that $\Phi_{\gamma}(p)$ is decreasing in $p>-1$.

    If $\Phi_{\gamma}(p_1)=\Phi_{\gamma}(p_2)$ for some $-1<p_1<p_2$, since $\Phi_{\gamma}(p)$ would not be strictly decreasing in $[p_1,p_2]$,
    then $\gamma$ would be linear, thus concluding the case of equality.
\end{proof}

Our next result is the aforementioned extension of \cite[Lemma 3.3]{AAGJV} to $p\in(-1,\infty)$.

\begin{proof}[Proof of Theorem \ref{thm:FuncBerwald>-1}]
Consider the probability measure on $\R^{n+1}$ given by $\displaystyle d\mu(x,t):=\frac{e^{-t}\chi_L(x,t)}{\int_L e^{-t}dtdx}dtdx$. Denote
$C_s(h)=\{(x,t)\in L\,:\,h(x,t)\geq s\}$
and define the function $I_h:[0,\infty)\to[0,\infty)$ as
$$
I_h(s):=\frac{1}{\int_L e^{-t}dtdx}\int_{C_s(h)}e^{-t}dtdx=\mu({C_s(h)}).
$$

$I_h$ is non-increasing,
$\displaystyle
I_h(0)=\mu(L)=1$ and since $h$ is concave, $I_h$ is log-concave (see \cite[Lemma 3.2]{AAGJV}).

 Observe that $\displaystyle(x,t)\in L$ if and only if $x\in K_t(f)$, which happens if and only if $\rho_{K_t(f)}(x)\ge 1$, and that, by Fubini's theorem, $\displaystyle\int_L\! e^{-t}dtdx=\!\!\int_0^\infty\!\!\! e^{-t}|K_t(f)|{dt}$. Now define $h_1:L\to[0,\infty)$ as
$$
h_1(x,t):=\sup\left\{s\in[0,\infty)\,:\, I_h(s)>\frac{1}{\rho^n_{K_t(f)}(x)}\right\}.
$$
$h_1$ has two important properties:

- $h$ and $h_1$ are equally distributed with respect to $\mu$, that is $I_{h_1}\equiv I_h$. In order to prove this, notice that for every $s\ge 0$, and every $(x,t)\in L$, we have that $h_1(x,t)>s$ if and only if $\rho^n_{K_t(f)}(x)>\frac{1}{I_h(s)}$ and so by Fubini's theorem,
$$
I_{h_1}(s)=\int_{C_s(h_1)}d\mu(x,t)=\int_0^\infty e^{-t}|K_t(f)|I_h(s)\frac{dt}{{\int_L e^{-t}dtdx}}
=I_h(s).
$$

- $h_1(r\rho_{K_t(f)}(u)u,t)$ does not depend on $t$ and $u$ since for any $r,t>0$, $u\in S^{n-1}$,
$$
h_1(r\rho_{K_t(f)}(u)u,t)=\sup\left\{s\in[0,\infty)\,:\, I_h(s)>r^n\right\}:=\gamma(r).
$$
Therefore, for any $p>0$,
\begin{eqnarray*}
\int_L h^p(x,t)d\mu(x,t)&=&\int_{\R^n}\int_0^\infty\chi_{\{h^p(x,t)\geq r\}}dr d\mu(x,t)
=\int_0^\infty I_h(r^\frac{1}{p})dr
\cr
&=&\int_0^\infty I_{h_1}(r^\frac{1}{p})dr=\int_L h_1^p(x,t)d\mu(x,t).\end{eqnarray*}

By Fubini's theorem and integrating in polar coordinates,
\begin{eqnarray*}
\int_L h_1^p(x,t)e^{-t}dxdt&=&\int_0^\infty e^{-t}\int_{K_t(f)}h_1^p(x,t)dxdt\cr &=&{n|B_2^n|}\int_0^\infty e^{-t}\int_{S^{n-1}}\int_0^{\rho_{K_t(f)}(u)}h_1^p(ru,t)r^{n-1}drd\sigma(u)dt\cr
&=&{n|B_2^n|}\int_0^\infty e^{-t}\int_{S^{n-1}}\int_0^{1}\gamma^p(r)\rho_{K_t(f)}^n(u)r^{n-1}drd\sigma(u)dt\cr
&=&{n}\int_0^\infty e^{-t}|K_t(f)|\int_0^{1}\gamma^p(r)r^{n-1}drdt
\end{eqnarray*}
and so,  since $ \int_0^\infty e^{-t}|K_t(f)|dt=\int_L e^{-t}dtdx$,
$$ \int_L h^p(x,t)d\mu(x,t)=n\int_0^1\gamma^p(r)r^{n-1}dr.$$

If $p<0$ the same equality holds. Indeed, we have
\begin{eqnarray*}
\int_L h^p(x,t)d\mu(x,t)&=&\int_0^\infty\int_{\R^n}\chi_{\{h(x,t)\leq r^\frac{1}{p}\}}drd\mu(x,t)=\int_0^\infty(1- I_h(r^\frac{1}{p}))dr\cr
&=&\int_0^\infty(1- I_{h_1}(r^\frac{1}{p}))dr=\int_L h_1^p(x,t)d\mu(x,t)
\end{eqnarray*} and we proceed as before.
If $p=0$ the equality is obviously true.

 Notice that since $I_h$ is log-concave the function $\gamma$ is non-increasing and for every $r_1,r_2\in[0,1]$,
$$
\gamma(r_1^{1-\lambda}r_2^\lambda)\geq(1-\lambda)\gamma(r_1)+\lambda \gamma(r_2).$$ If we denote  $\gamma_1(r)=\gamma(e^{-r/n})$ the previous statement means that $\gamma_1$ is non-decreasing and concave in $[0,\infty)$ and we have
$$
\int_L h^p(x,t)d\mu(x,t)=n\int_0^1\gamma^p(r)r^{n-1}dr=\int_0^\infty \gamma_1^p(r) e^{-r}dr.
$$

We can apply now Lemma \ref{1dim} to the function $\gamma_1$ and conclude that
$$
\left(\frac{1}{\Gamma(1+p)\int_L e^{-t}dtdx}\int_L h^p(x,t)e^{-t}dtdx\right)^\frac{1}{p}
$$
is non-decreasing in $(-1,\infty)$.
\end{proof}

\section{Proof of functional Zhang's inequality}\label{sec:Proof}

In this section we will give the proof of the functional version of Zhang's inequality \eqref{eq:FunctionalZhang}. For any $g\in\mathcal F(\mathbb R^n)$
such that $g(0)>0$ and $p>0$, we will consider the following important family of convex bodies, which was introduced  by K. Ball in \cite[pg.~74]{B}. We denote
$$
\widetilde{K}_p(g):=\left\{x\in\R^n\,:\, \int_0^\infty g(rx)r^{p-1}dr\geq\frac{g(0)}{p}\right\}.
$$
It follows from the definition that the radial function of $\widetilde{K}_p(g)$ is given by
$$
\rho^p_{\widetilde{K}_p(g)}(u)=\frac{1}{g(0)}\int_0^\infty pr^{p-1}g(rx)dr.
$$

\begin{rmk}\label{rmk:InclusionBallsBodies}
	It is well known (cf.~\cite[Proposition 2.5.7]{BGVV}) that for any $g\in\mathcal F(\mathbb R^n)$ such that $\Vert g\Vert_\infty=g(0)$ and $0<p\leq q$,
	$$
	\frac{\Gamma(1+p)^\frac{1}{p}}{\Gamma(1+q)^\frac{1}{q}}\widetilde{K}_q(g)\subseteq \widetilde{K}_p(g)\subseteq \widetilde{K}_q(g).
	$$
\end{rmk}

We will make use of the following well known relation (cf.~\cite{B}) between the Lebesgue measure of $\widetilde{K}_n(g)$ and the integral of $g$.

\begin{lemma}[\cite{B}]\label{VolumeAndIntegral}
	Let $g\in\mathcal F(\mathbb R^n)$ be such that $g(0)>0$. Then
	$$
	|\widetilde{K}_n(g)|=\frac{1}{g(0)}\int_{\R^n}g(x)dx.
	$$
\end{lemma}

For any $f\in\mathcal F(\mathbb R^n)$, we collect below the properties of its covariogram functional $g_f$, whose proof can be found in \cite[Lemma 2.1]{ABG}.
\begin{lemma}\label{lem:g(x)}
Let $f\in\mathcal F(\mathbb R^n)$. Then the function $g_f:\R^n\to\R$ defined by
$$
g_f(x)=\int_{\R^n}\min\left\{\frac{f(y)}{\Vert f\Vert_\infty},\frac{f(y-x)}{\Vert f\Vert_\infty}\right\}dy
$$
is even, log-concave, $0\in\mathrm{int}(\mathrm{supp}\ g_f)$ with $\displaystyle \Vert g_f\Vert_\infty=g_f(0)=\int_0^\infty e^{-t}|K_t(f)|dt=\int_{\R^n}\frac{f(x)}{\Vert f\Vert_\infty}dx>0$, and $\displaystyle{\int_{\R^n}g_f(x)dx=\int_{\R^n}\int_{\R^n}\min\left\{\frac{f(y)}{\Vert f\Vert_\infty},\frac{f(x)}{\Vert f\Vert_\infty}\right\}dydx}$.
\end{lemma}

In the particular case of $g_f$ as in Lemma \ref{lem:g(x)}, we can provide an alternative definition for $\widetilde{K}_p(g_f)$ in terms of its radial function that will allow us to obtain the polar projection body of $f$ as a limiting case of this expression when $p$ tends to $-1$.

\begin{lemma}\label{rhoK_g}
	Let $f\in\mathcal F(\mathbb R^n)$ and let $g_f:\R^n\to\R$ be the function
	$$
	g_f(x)=\int_0^\infty e^{-t}|K_t(f)\cap(x+ K_t(f))|dt.
	$$
	Then, for any $u\in S^{n-1}$ and
$p>0$,
	$$
	\rho_{\widetilde{K}_p(g_f)}^p(u)=\frac{1}{(p+1)\int_{\R^n}\frac{f(x)}{\Vert f\Vert_\infty}dx}\int_0^\infty e^{-t}\int_{P_{u^\perp}K_t(f)}|K_t(f)\cap(y+\langle u\rangle)|^{p+1}dydt$$
\end{lemma}

\begin{rmk}
Notice that the right hand side in the equality above
is defined for $p>-1$ and that, since $(p+1)\Gamma(1+p)=\Gamma(2+p)$, if $p\to -1^+$ then
$$
\frac{1}{(p+1)\Gamma(1+p)\int_{\R^n}\frac{f(x)}{\Vert f\Vert_\infty}dx}\int_0^\infty e^{-t}\int_{P_{u^\perp}K_t(f)}|K_t(f)\cap(y+\langle u\rangle)|^{p+1}dt\to\frac{\Vert u\Vert_{\Pi^*(f)}}{2\Vert f\Vert_1}.
$$
\end{rmk}

\begin{proof}[Proof of Lemma \ref{rhoK_g}]
	By Lemma \ref{lem:g(x)}, $g_f(0)>0$ and
\begin{eqnarray*}
	&&\!\!\!\!\!\!\!\!\!\!\rho_{\widetilde{K}_p(g_f)}^p(u)=\cr
	&=&\frac{p}{g_f(0)}\int_0^\infty r^{p-1}g_f(ru)dr \cr
    &=&\frac{1}{g_f(0)}\int_0^\infty pr^{p-1}\int_0^\infty e^{-t}|K_t(f)\cap(ru+ K_t(f))|dtdr\cr
	&=&\frac{1}{g_f(0)}\int_0^\infty e^{-t}\int_0^{\rho_{K_t(f)-K_t(f)}(u)}\!\!\!pr^{p-1}|K_t(f)\cap(ru+ K_t(f))|drdt\cr
	&=&\frac{1}{g_f(0)}\int_0^\infty\!\!\!\! e^{-t}\int_0^{\rho_{K_t(f)-K_t(f)}(u)}\!\!\!\!pr^{p-1}\int_{P_{u^\perp}K_t}\!\!\!\!\max\{|K_t(f)\cap (y+\langle u\rangle)|-r,0\}dydrdt\cr
	&=&\frac{1}{g_f(0)}\int_0^\infty e^{-t}\int_{P_{u^\perp}K_t(f)}\int_0^{|K_t(f)\cap (y+\langle u\rangle)|}\!\!\!\!\!pr^{p-1}\left(|K_t(f)\cap (y+\langle u\rangle)|-r\right)\ drdydt\cr
	&=&\frac{1}{(p+1)g_f(0)}\int_0^\infty e^{-t}\int_{P_{u^\perp}K_t(f)}|K_t(f)\cap (y+\langle u\rangle)|^{p+1}dydt.\cr
\end{eqnarray*}
	
\end{proof}

\begin{proof}[Proof of inequality \eqref{eq:FunctionalZhang}]
Let $u\in S^{n-1}$ and define on $h:L\to[0,\infty)$ the function
$$
h(x,t)=|K_t(f)\cap\{(x,t)+\lambda u\,:\,\lambda\geq0\}|,
$$
where $L$ is the epigraph of $-\log\frac{f}{\Vert f\Vert_\infty}$.
Since $L$ is convex, $h$ is concave. For any $p>-1$ we have,
$$
\frac{1}{(p+1)}\int_0^\infty e^{-t}\int_{P_{u^\perp}K_t(f)}|K_t(f)\cap(y+\langle u\rangle)|^{p+1}dxdt=\int_0^\infty\int_{K_t(f)} e^{-t}h(x,t)^p dxdt.
$$
Therefore, by Theorem \ref{thm:FuncBerwald>-1}, for every $-1<p<0$,
$$
\frac{1}{(p+1)\Gamma(1+p)\int_{\R^n}\frac{f}{\Vert f\Vert_\infty}}\int_0^\infty\!  e^{-t}\! \int_{P_{u^\perp}K_t(f)}\! \! \!\! \!  |K_t(f)\cap(y+\langle u\rangle)|^{p+1}dt dx\leq$$ $$\! \leq\! \left(\frac{1}{(n+1)n!\int_{\R^n}\frac{f}{\Vert f\Vert_\infty}}\int_0^\infty\!\! e^{-t}\int_{P_{u^\perp}K_t(f)}\!\!\! \! \! \! |K_t(f)\cap(y+\langle u\rangle)|^{n+1}dt dx\right)^\frac{p}{n}\! \! \!
=\frac{\rho_{\widetilde{K}_n(g_f)}(u)^{p}}{n!^\frac{p}{n}}.$$
Taking limit as $p\to-1$ and by Lemma \ref{rhoK_g} we obtain
$$
\rho_{\widetilde{K}_n(g_f)}(u)\leq 2(n!)^\frac{1}{n}\Vert f\Vert_1\rho_{\Pi^*(f)}(u),
$$
that is,
$$
\widetilde{K}_n(g_f)\subseteq2(n!)^\frac{1}{n}\Vert f\Vert_1\Pi^*\left(f\right).
$$
Taking Lebesgue measure and using Lemmas \ref{lem:g(x)} and \ref{VolumeAndIntegral} we obtain inequality \eqref{eq:FunctionalZhang}.
\end{proof}

\end{document}